\documentclass[10pt, oneside,reqno]{amsart}
\usepackage[english]{babel}
\usepackage{indentfirst}
\usepackage[T1]{fontenc}
\usepackage[utf8]{inputenc}
\usepackage{paralist}
\usepackage{amssymb}
\usepackage{amsthm}
\usepackage{amsmath}
\usepackage{amscd}
\usepackage{graphicx}
\usepackage[colorlinks=true]{hyperref}
\usepackage[margin=1.4in]{geometry}
\usepackage{cite}
\usepackage{lineno} %% adds numbers to the lines
\theoremstyle{plain} 
\newtheorem{theorem}{Theorem}[section] 
 
\newtheorem{lemma}[theorem]{Lemma}

\newtheorem{defn}[theorem]{Definition}
\newtheorem{remark}[theorem]{Remark}
\numberwithin{equation}{section}

\newcommand{\R}{\ensuremath{\mathbb{R}}}
\newcommand{\Rn}{\ensuremath{\mathbb{R}^n}}
\newcommand{\N}{\ensuremath{\mathbb{N}}}

\newcommand{\eps}{\ensuremath{\varepsilon}}
\newcommand{\frlap}{\ensuremath{(-\Delta)^s}}
\newcommand{\al}{\ensuremath{ &\;}}
\newcommand{\lr}[1]{\left( #1 \right)}
\newcommand{\lrq}[1]{\left[ #1 \right]}
\newcommand{\eqlab}[1]{\begin{equation}  \begin{aligned}#1 \end{aligned}\end{equation}}
\newcommand{\bgs}[1]{\begin{equation*} \begin{aligned}#1\end{aligned}\end{equation*}}
 
  \newcommand{\sys}[2][]{\begin{equation*}#1  \left\{\begin{aligned}#2\end{aligned}\right.\end{equation*}}

\begin{document}
\nocite*
\date{}
\author{Claudia Bucur}

\address{Claudia Bucur: Dipartimento di Matematica\\ Universit\`a degli Studi di Milano \\ Via Cesare Saldini, 50 \\ 20100, Milano-Italy}

\email{claudia.bucur@unimi.it}

\author{Aram L. Karakhanyan}
\address{Aram L. Karakhanyan: Maxwell Institute for Mathematical Sciences and
School of Mathematics, University of Edinburgh,
James Clerk Maxwell Building, Peter Guthrie Tait Road,
Edinburgh EH9 3FD,
United Kingdom.
}
\email{aram.karakhanyan@ed.ac.uk}
\email{}
\thanks{The research of the second author is partially supported by EPSRC grant}

\title[]{Potential theoretic approach to Schauder estimates for the fractional laplacian}
\begin{abstract}
We present an elementary approach for the proof of  Schauder estimates for the equation
$(-\Delta)^s u(x)=f(x), \,0<s<1$, with $f$ having a modulus of continuity $\omega_f$, 
based on the Poisson representation formula and 
dyadic ball approximation argument. We give the explicit modulus of continuity of 
$u$ in balls $B_r(x)\subset \R^n$ in terms of $\omega_f$.
\end{abstract}

\maketitle

\tableofcontents
%--------------------------------------
% SECTION
%--------------------------------------
\section{Introduction}

Let $f$ be a given H\"{o}lder continuous function and $u$ solving
$(-\Delta)^s u(x)=f(x)$ with a fixed $s\in (0,1)$.
%, with $f$ having a modulus of continuity $\omega_f$, 
We want to study the regularity of $u$ using a very simple method 
based on the Poisson representation formula
and dyadic ball approximation argument. 
In order to formulate our results it is convenient to introduce some notations and 
basic knowledge on the fractional Laplacian and on some related kernels. 
For further details on the fractional Laplacian and applications, see \cite{nonlocal,galattica}.  
Also, for regularity up to the boundary of weak solutions of the Dirichlet problem, see the very nice paper \cite{oton}.

\smallskip 

In what follows we assume that $n\geq 2$.
We denote by $\mathcal{S}$ the Schwartz space of rapidly decreasing functions, defined as follows 
\[\mathcal S = \left\{ u \in C^\infty(\Rn) \quad \mbox{s.t. for any } \alpha,\, \beta \in \N_0^n,  \quad \sup_{x\in\Rn} |x^\beta D^\alpha u(x) |<\infty\right\}. \]
%We set the usual notations 
%\[ \mathcal{F} f(\xi)=\int_{\Rn} f(x)e^{-ix\cdot\xi} \, dx \quad \mbox{and}\quad  \mathcal{F}^{-1} f(x)=\int_{\Rn} f(\xi)e^{ix\cdot\xi} \, d\xi\]
%for the Fourier and the inverse Fourier transform.  
%We denote by $\mathcal S'$ its topological dual and by $\langle \cdot, \cdot\rangle$ the duality pairing.

Let $s\in (0,1)$ be fixed. We have the following integral definition.
\begin{defn} The fractional Laplacian of $u\in \mathcal S$ is defined for any $x\in \Rn$ as
\bgs{  \frlap u(x) =&\;   C(n,s) \mbox{P.V.}\int_{\Rn} \frac{u(x) -u(x-y)}{|y|^{n+2s}} \, dy\\
				=&\; C(n,s)  \lim_{\eps \to 0}  \int_{\Rn\setminus B_\eps}  \frac{u(x) -u(x-y)}{|y|^{n+2s}} \, dy,}
				where $P.V.$ stand for ``{in the principal value sense}'', as defined in the last line, and $C(n,s)$ is a  constant depending only on $n$ and $s$. In what follows we call such constants dimensional.
\end{defn} 

With a change of variables, one obtains an equivalent representation, given by
	\[ \frlap u(x)=\frac{C(n,s)}2 \int_{\Rn} \frac{ 2 u(x) -u(x-y)-u(x+y)}{|y|^{n+2s}} \, dy.\]
	In this expression the principle value may be omitted if $u \in L_s^1 (\Rn) \cap C^{2s+\eps} (B_r(x))$ for a small $r>0$. Here, the space $L_s^1(\Rn)$ is the weighted $L^1$ space, defined as	 \[L^1_s (\Rn):=\left\{ u \in L^1_{\text{loc}}(\Rn)\; \mbox{ s.t. } \; \int_{\Rn}\frac{ |u(x)|}{1+|x|^{n+2s}} \, dx <\infty\right\}.\] Moreover, we denote by $C^{2s+\eps}$ for small $\eps$ the H\"{o}lder space $C^{0,2s+\eps}$ for $s<1/2$ and $C^{1,2s+\eps-1}$ for $s\geq 1/2$.
	
\smallskip 

The main result we prove in this paper is a Schauder type estimate for bounded solutions of the equation $(-\Delta)^s u=f$ in $B_1$. Given  $f\in C^{0,\alpha}(B_1)\cap C(\overline B_1)$, then on the half ball $u$ has the regularity of $f$ increased by $2s$. More precisely:
\begin{theorem} \label{schthm}Let $s\in(0,1)$,  $\alpha <1$ and $f\in C^{0,\alpha}(B_1)\cap C(\overline B_1)$ be a given function with modulus of continuity $\omega(r):=\sup_{|x-y|<r} |f(x)-f(y)|$.
 Let  $u\in  L^{\infty}(\Rn)\cap C^1(B_1)$ be a pointwise solution of $\frlap u=f$ in $B_1$. Then for any $x,y \in B_{1/2}$ and denoting $\delta:=|x-y|$ we have that for $ s\leq 1/2$ 
  \eqlab{\label{mm1}    |u(x)-u(y)|\leq &\; C_{n,s}  \bigg( \delta  \|u\|_{L^{\infty}(\Rn\setminus B_1)} +\delta \sup_{\overline B_1}|f| + \int_0^{c \delta} \omega(t)t^{2s-1}\, dt + \delta \int_{\delta}^1  \omega(t)t^{2s-2}\, dt\bigg) }
  while for $s>1/2$ 
  \eqlab{\label{mm} 
 |Du(x)-Du(y)|\leq &\; C_{n,s}  \bigg(\delta  \|u\|_{L^{\infty}(\Rn\setminus B_1)} +\delta \sup_{\overline B_1}|f| + \int_0^{c \delta} \omega(t)t^{2s-2}\, dt+ \delta \int_{\delta}^1  \omega(t)t^{2s-3}\, dt\bigg) ,}
 where $C_{n,s}$ and $ c$ are positive dimensional constants.
\end{theorem}

 There are other approaches to prove Schauder estimates for the fractional order operators with more general kernels 
 see  \cite{Dong} and references therein. Here we follow the one proposed by Xu-Jia Wang in \cite{wang} which is based only on the 
 higher order derivative estimates, that we state here in Lemma \ref{estimate1} and on a maximum principle, given in Lemma \ref{estimate2}. 
  \smallskip

 We prove these estimates using some kernels related to the fractional Laplacian (see Chapter I.6 in \cite{Landkof} or \cite{bucur} for more details), that we introduce now. 
We take $r>0$ and introduce the fractional Poisson kernel on the ball. For any $ x \in B_r$ and any $ y \in \Rn \setminus \overline{B}_r$ we define
	\begin{equation} \label{poisson}
	 P_r(y,x) :=  c(n,s) \Bigg (\frac {r^2-|x|^2}{|y|^2-r^2}\Bigg)^s \frac {1}{|x-y|^n}, 
	\end{equation} 
where $c(n,s)$ is a dimensional constant given in such a way that
\[ \int_{\Rn\setminus B_r} P_r(y,x) \, dy =1.\] Then (see Theorem 2.10 in \cite{bucur}) one has for $u \in L^{\infty}(\Rn)\cap C(\Rn\setminus B_r)$ that the equation 
\[ \frlap u(x) = 0 \mbox{ in } B_r\] has a pointwise solution given by
\eqlab{ \label{poissonrepr} u(x) = \int_{{\Rn}\setminus B_r} P_r(y,x) u(y)\, dy.}

We recall also a representation formula for the equation $\frlap u=f$ in $\Rn$.
For any $x\in \Rn \setminus \{0\}$ we define
   	\begin{equation} 
		\Phi(x) := a(n,s){|x|^{-n+2s}},
		\label{fondsol}
	\end{equation}
where $a(n,s)$ is a dimensional constant.
The function $\Phi$ plays the role of the fundamental solution of the fractional Laplacian, i.e. in the distributional sense
\[ \frlap \Phi =\delta_0,\] where $\delta_0$ is the Dirac delta evaluated at $0$ (see Theorem 2.3 in \cite{bucur} for the proof). For a function $f\in C^{0,\eps}_c(\Rn),$ where  $\eps>0$ is a small quantity, we define 
\eqlab{\label{fondsolrepr} u(x):=f*\Phi(x)=a(n,s)\int_{\Rn}\frac{f(y)}{|x-y|^{n-2s}}\, dy.} Then (see Lemma 2.6 in \cite{bucur}) we have that $u \in L_s^1(\Rn)$. Moreover $u\in C^{2s+\eps}(\Rn)$ as we see in the next Section \ref{newton}. This, together with Theorem 2.8 in \cite{bucur} says that $u$ is pointwise solution of
\eqlab{ \label{fondsolcon} \frlap u=f \quad \mbox{in} \quad \Rn.} 
 
 \medskip 
 
%--------------------------------------
% SUBSECTION Active scalars
%--------------------------------------
%\section{Active scalar equation}
One of the motivations to study \eqref{fondsolcon} comes from 
the active scalars (see \cite{Constantin}). The 2D incompressible Euler equation  
\begin{eqnarray}
\left\{
\begin{array}{lll}
\omega_t+v\nabla \omega=0\\
v=(\partial_2\psi, -\partial_1\psi)\\
\omega=\Delta\psi
\end{array}
\right.
\end{eqnarray}
is one of the well-known active scalar equations. Here $v$ is the 
velocity, $\omega$ the vorticity, $\psi$ the stream function. 

\smallskip

The uniqueness was proved  by Yudovich (see \cite{russo}) under the condition that 
$\omega(t)\in L^{\infty}(0, T; L^1(\R^2))\cap  L^{\infty}(0, T; L^\infty(\R^2))$. 
Observe that by the Biot -Savart law one has that $v=k*\omega$, where 
$$k(x)=\frac{x^\perp}{2\pi |x|^2}.$$ 
Clearly $k\in L^p_{loc}(\R^2), \,1\leq p<2$
and $k\in L^q(\R^2),\, q>2$ near infinity, implying that one must assume 
$\omega\in L^{p_o}(\R^2)\cap L^{q_o}(\R^2), \,p_o<2<q_o$ to make sure that $v=k*\omega$
is well defined.  In particular $p_o=1,\, q_o=\infty$ will do.

\medskip 

A generalization of the 2D Euler equation is the 
quasigeostrophic active scalar 
\begin{eqnarray}
\left\{
\begin{array}{lll}
\omega_t+v\nabla \omega=0\\
v=(\partial_2\psi, -\partial_1\psi)\\
-\omega=(-\Delta)^\frac12 \psi
\end{array}
\right.
\end{eqnarray}
or more generally when one takes $-\omega=(-\Delta)^s \psi, \,0<s<1$. 
Thus this leads to the study of $k_s*(\Delta)^{-s}\omega$
where $n=2$ and 
$$k_s(x)=\nabla^\perp \frac{C_{n,s}}{|x|^{n-2\sigma}}.$$
We see that the regularity of the stream function can be 
concluded from that of $\omega$ via Schauder estimates.
\medskip 

%-------------------------------------
% Schauder history
%-------------------------------------$v_t+v\nabla v=\nabla p$ writing it in terms of vorticity $\omega$ 

%--------------------------------------
% SECTION
%--------------------------------------
 \section{H\"{o}lder estimates for the Riesz  potentials}\label{newton}
In the next Lemma, we establish that given a bounded function with bounded support, its convolution with the function $\Phi$ defined in \eqref{fondsol} is H\"{o}lder continuous. 
 \begin{lemma}\label{lem1}
 Let $s\in (0,1/2)\cup (1/2,1)$ be fixed. Let $\Omega \subseteq \Rn$ be a bounded set, the function $f\in L^{\infty}(\Rn)$ be supported in $\Omega$ and $u$ be defined as
 \eqlab{  \label{lem1f2}u(x):= \int_{\Rn} \frac{f(y)}{|x-y|^{n-2s}} \, dy.}Then 
 $u\in C^{0,2s}(\Rn)$ for $s<1/2$ and $u\in C^{1,2s-1}$ for $s>1/2$. 
 \end{lemma}

The proof of this Lemma takes inspiration from \cite{russo}, where some bounds are obtained in the case $s=1/2$. Check also Lemma 3.1 in \cite{samko} for other considerations.
 
\begin{proof} Let $s<1/2$ be fixed. We consider $x_1,x_2 \in \Rn$ and denote by $\delta:=|x_1-x_2|$. We notice that in the course of the proof, the constants may change value from line to line. We have that
\eqlab{ \label{lem1eq1} |u(x_1)&-u(x_2)| \\\leq \al \int_{\Omega  }|f(y)| \bigg| \frac{1}{|x_1-y|^{n-2s}} -\frac{1}{|x_2-y|^{n-2s}} \bigg| \, dy \\
\leq \al \|f\|_{L^{\infty}(\Rn)} \Bigg[ \int_{\Omega \cap \{ |x_1-y|\leq 2\delta\} }\frac{dy}{|x_1-y|^{n-2s}} +  \int_{\Omega \cap \{ |x_1-y|\leq 2\delta\}} \frac{dy}{|x_2-y|^{n-2s}}  \\ \al 
+ 
 \int_{\Omega \setminus \{ |x_1-y|\leq 2\delta\}} \bigg| \frac{1}{|x_1-y|^{n-2s} }-\frac{1}{|x_2-y|^{n-2s}} \bigg|\, dy \Bigg]\\
 =:\al\|f\|_{L^{\infty}(\Rn)}\big(  I_1 +I_2+I_3 \big).  }
By passing to polar coordinates we have that 
 \[ I_1 \leq C_n  \int_0^{2\delta} \rho ^{2s-1} \, d\rho = C_{n,s} \delta^{2s}.\]
 At the same manner, noticing that $|x_2-y|\leq |x_2-x_1|+|x_1-y|\leq 3\delta$ we have 
 \[ I_2  \leq C_n   \int_0^{3\delta}\rho^{2s-1}\, d\rho = C_{n,s} \delta^{2s}.\]
 For $I_3$, we see that $|x_2-y|\geq|x_1-y|- |x_1-x_2|\geq \delta$. The function $|x-y|^{2s-n}$ is differentiable at each point of the segment $x_1x_2$ and using the mean value theorem we have that for $x^\star$ on the segment $x_1x_2$ 
 \[\left| \frac{1}{|x_1-y|^{n-2s}} -\frac{1}{|x_2-y|^{n-2s}} \right| \leq C_n \frac{|x_1-x_2|}{|x^\star-y|^{n-2s+1}}    =C_n \delta \frac{1}{|x^\star-y|^{n-2s+1}} .\]
It follows that
 \[ I_3\leq C_n \delta  \int_{\Omega \setminus \{ |x_1-y|\leq 2\delta\}}  \frac{1}{|x^\star-y|^{n-2s+1}}\, dy.\]
 Since $|x^\star-x_1| \leq \delta \leq \frac{1}2 |x_1-y|$, we have that $|x^\star-y|\geq \frac{1}2 |x_1-y|$. Passing to polar coordinates, since $2s<1$, we get that
 	\eqlab{ \label{forln1} I_3\leq &\;C_n\delta   \int_{2\delta}^{\infty}  \rho ^{2s-2}\, d\rho = C_{n,s} \delta^{2s}  .}
 	 	By inserting these bounds into \eqref{lem1eq1} we obtain that
\eqlab{\label{lem1eq2} |u(x_1)-u(x_2)| \leq C \delta^{2s},}where $C=C({n,s,f}) $ is a positive constant.  
To prove the bound for $s>1/2$, thanks to Lemma 4.1 in \cite{trudy} we have that
\eqlab{\label{fundsolderiv} D u(x)=\int_{\Omega} D \Phi(x-y) f(y)\, dy=\int_{\Omega} \frac{f(y) }{|x-y|^{n-2s+1}}\, dy.}
The proof then follows as for $s<1/2$, and one gets that
\bgs{ |D u(x_1) \al-D u(x_2)| 	\\ \leq \al \|f\|_{L^{\infty}(\Rn)} \Bigg[ \int_{\Omega \cap \{ |x_1-y|\leq 2\delta\} }\frac{dy}{|x_1-y|^{n-2s+1}} +  \int_{\Omega \cap \{ |x_1-y|\leq 2\delta\}} \frac{dy}{|x_2-y|^{n-2s+1}} \\ \al 
+ 
 \int_{\Omega \setminus \{ |x_1-y|\leq 2\delta\}}  \bigg| \frac{1}{|x_1-y|^{n-2s+1} }-\frac{1}{|x_2-y|^{n-2s+1}}\bigg|\, dy \\
 \leq \al C  \delta^{2s-1},} where $C=C({n,s,f})$ is a positive constant. 
 This concludes the proof of the Lemma.
\end{proof}
\begin{remark} 
On $\Omega$ one has the following bounds. For $x_1,x_2\in \Omega$
 \sys [|u(x_1)-u(x_2) |\leq] {&\; C |x_1-x_2|\Big(1+	|x_1-x_2|^{2s-1}\Big)    &\mbox{ for } \quad &s<  1/2\\
 &\;C|x_1-x_2| \Big(1+\big|\ln|x_1-x_2|\big|  \Big) & \mbox{ for } \quad &s= 1/2 ,} and
 \bgs{ |Du(x_1)-Du(x_2)|\leq C|x_1-x_2| \Big(1+|x_1-x_2|^{2s-2} \Big)&\mbox{\; for } \quad &s>  1/2,
 }
where $C=C(n,s,f, \Omega) $ is a positive constant depending on the dimension of the space, the fractional parameter $s$, the $L^{\infty}$ norm of $f$ and the diameter of $\Omega$.  

To see these, it is enough to modify \eqref{forln1} as follows. Denoting $R:=\mbox{diam } \Omega$, since $|x_1|<R$ for $s<1/2$ we get that 
\bgs{ I_3\leq &\;C_n \delta   \int_{2\delta}^{2R }  \rho ^{2s-2}\, d\rho = C_{n,s} \delta( \delta^{2s-1} - R^{2s-1})   \leq C_{n,s} \delta(\delta^{2s-1}+1).}
For $s=1/2$, we have $I_1, I_2$ are bounded by $C_{n,s} \delta$ and
 	 	\[I_3  \leq C_{n,s} \delta \ln \frac{R}{\delta} \leq C_{n,s,R} (1+|\ln \delta|).\]
 	 	From this and the bounds established in the proof of Lemma \ref{lem1}, the estimates in this remark plainly follow.
 	 	\end{remark} 
We prove that the Riesz potential is $C^{2s+\eps}(\Rn)$ when $f\in C^{0,\eps}_c(\Rn)$. 
The proof is included for completeness. 

\begin{lemma}  \label{lem12} Let $s\in (0,1)$ be fixed. Let $f\in C_c^{0,\eps}(\Rn)$ be a given function and $u$ be defined as
 \eqlab{ u(x):= \int_{\Rn} \frac{f(y)}{|x-y|^{n-2s}} \, dy.}Then 
 $u\in C^{2s+\eps}(\Rn)$. 
\end{lemma}
\begin{proof}
Let $R>0$ such that $\mbox{supp } f\subseteq B_R$ and let $s<1/2$. Then taking $x_1, x_2 \in \Rn$ and denoting by $\delta:=|x_1-x_2|$ we have that
\bgs{ |u(x_1)&-u(x_2)| \\\leq \al \int_{B_R \cap \{ |x_1-y|\leq 2\delta\} }\frac{|f(y)-f(x_1)|}{|x_1-y|^{n-2s}} \, dy+  \int_{B_R \cap \{ |x_1-y|\leq 2\delta\}} \frac{|f(y)-f(x_1)|}{|x_2-y|^{n-2s}} \, dy\\ \al 
+ 
 \int_{B_R \setminus \{ |x_1-y|\leq 2\delta\}} |f(y)-f(x_1)| \bigg| \frac{1}{|x_1-y|^{n-2s} }-\frac{1}{|x_2-y|^{n-2s}}\bigg|\, dy  \\ \al+ |f(x_1)|\bigg|\int_{B_R} \lr{ \frac{1}{|x_1-y|^{n-2s} }-\frac{1}{|x_2-y|^{n-2s}}} dy \bigg|\\
 =\al I_1 +I_2+I_3 +I_4.  }
 Since $f$ is H{\"o}lder continuous, we have that for $C>0$
 \[ |f(y)-f(x_1)|\leq C |y-x_1|^{\eps}.\] Noticing that in the next computations the constants may change value form line to line, we obtain that
 \bgs{ & I_1 \leq C  \int_{B_R \cap \{ |x_1-y|\leq 2\delta\} } |x_1-y|^{-n+2s+\eps}\, dy = C_{n,s} \delta^{2s+\eps},\\
 &I_2 \leq  C\int_{B_R \cap \{ |x_1-y|\leq 2\delta\} } |x_1-y|^\eps |x_2-y|^{-n+2s}\, dy \leq C_{n,s} \delta^{2s+\eps}}
 since $|x_2-y| \leq 3\delta$ and for $x^\star $ on the segment $x_1\, x_2$, recalling that $s<1/2$ 
   \bgs{  I_3 \leq  C  \delta \int_{B_R \setminus   \{ |x_1-y|\leq 2\delta\} } \frac{|x_1-y|^{\eps}}{|x^\star -y|^{n-2s+1}}  \, dy  \leq C_{n,s} \delta^{2s+\eps}, }
   given that $|x^\star -y|\geq \frac{1}2 |x_1-y|$.

Now, if $x_1\notin B_R$ or $x_2\notin B_R$ (it is enough in this latter case to replace $x_1$ with $x_2$ in the above computations), then we are done. Else, for  $x_1,x_2 \in B_R$, suppose that $\mbox{dist}(x_1,\partial B_R) \geq  \mbox{dist}(x_2,\partial B_R)  $ and take $p\in \partial B_R$ (hence $f(p)=0$) such that $\mbox{dist} (x_1,\partial B_r)=|x_1-p|$. So
 \[|f(x_1)|=|f(x_1)-f(p)|\leq C  |x_1-p|^{\eps} \]  and we distinguish two cases. When
 \[  (1) \quad |x_1-x_2| \geq \frac{1}2|x_1-p|\]   we use the result in Lemma \ref{lem1}, observing that from \eqref{lem1eq1} and \eqref{lem1eq2} we have that
 \[ \int_{B_R }\bigg| \frac{1}{|x_1-y|^{n-2s}}-\frac{1}{|x_2-y|^{n-2s}}\bigg|  \, dy \leq C |x_1-x_2|^{2s} ,\]
 hence 
 \[ I_4\leq  C|x_1-p|^{\eps}  |x_1-x_2|^{2s} \leq C \delta^{2s+\eps}.\] On the other hand, when   \[ (2)\quad |x_1-x_2| \leq \frac{1}2 |x_1-p|\]  we use the following bound (see Lemmas 2.1 and 3.5 in \cite{samko})
 \eqlab{\label{sdom} \bigg|\int_{B_R }|x_1-y|^{2s-n}-|x_2-y|^{2s-n} \, dy\bigg| \leq \frac{|x_1-x_2|}{\max\{\mbox{dist}(x_1,\partial B_R) ,\mbox{dist}(x_2,\partial B_R) \}^{1-2s} }.} Since $2s+\eps-1 <0$ we get that
\bgs{ I_4 \leq C \delta |x_1-p |^{2s-1+\eps} \leq C_{n,s} \delta^{2s+\eps}.}
 This concludes the proof of the Lemma for $s<1/2$. 
 In order to prove the result for $s\geq1/2$, one considers the formula \eqref{fundsolderiv} and iterates the computations above.
   \end{proof}
The interested reader can see Theorem 4.6 in \cite{samko}, where the result given here in Lemma \ref{lem12} is proved for $u$ defined as
\[ u(x)= \int_{\Omega} \frac{f(y)}{|x-y|^{n-2s}} \, dy,\] where $\Omega$ is a domain with the $s$-property (see Definition 3.3 therein). In particular, these domains are defined such that they satisfy a bound of the type given in \eqref{sdom}, while the ball is the typical example of this type.  
 
%--------------------------------------
% SECTION
%--------------------------------------
 \section{Some useful estimates}
 
 In this section we introduce some useful estimates, using the representation formulas \eqref{poissonrepr} and \eqref{fondsolrepr}. The interested reader can also check \cite{fall}, where Cauchy-type estimates for the derivatives of $s$-harmonic functions are proved using the Riesz and Poisson kernel.
 
 We fix $r>0$.
 \begin{lemma}\label{estimate1} Let $u\in  L^\infty(\Rn)\cap C(\Rn\setminus B_r) $ be such that $\frlap u(x)=0$  for any $x$ in $B_r$. Then for any $\alpha \in \N^n_0$
 	\eqlab{\label{estimate1} \|D^\alpha u \|_{L^{\infty}(B_{r/2})} \leq c r^{-|\alpha|} \|u\|_{L^\infty(\Rn\setminus B_r)},} where $c=c(n,s,\alpha)$ is a positive constant.
 \end{lemma}
 \begin{proof}
 We notice that it is enough to prove \eqref{estimate1} for $r=1$, i.e.
 \eqlab{\label{estimate16} \|D^\alpha u \|_{L^{\infty}(B_{1/2})} \leq c \|u\|_{L^\infty(\Rn\setminus B_1)}.}   Indeed, if \eqref{estimate16} holds, then by rescaling, namely letting $y=rx$ and $v(y)=u(x)$   for $x\in B_1$, we have that  $D^\alpha u(x) =r^{|\alpha|}D^\alpha v(y) $. Hence $r^{|\alpha|}| D^\alpha v(y)|=|D^\alpha u(x)|\leq c\|u\|_{L^{\infty}(\Rn\setminus B_1) }= c\|v\|_{L^\infty(\Rn\setminus B_r)}$  and one gets the original estimate for any $r$.
 
 We use the representation formula given in \eqref{poissonrepr}. By inserting definition \eqref{poisson}, we have that in $B_1$
 \bgs{ u(x) = &\;\int_{\Rn \setminus B_1} u(y) P_1(y,x)\, dy\\
 =&\; c(n,s) \int_{\Rn \setminus B_r} u(y)\frac{(1-|x|^2)^s}{(|y|^2-1)^s} \frac{dy}{|x-y|^n}.}
 Let $x\in B_{1/2}$. We take the j$^{th}$ derivative of $u$ and have that
  \bgs{ &\;  D_j u(x) \\ =&\; c(n,s) \int_{\Rn \setminus B_1} u(y)D_j \lrq{ \frac{(1-|x|^2)^s}{(|y|^2-1)^s} \frac{1}{|x-y|^n}}\, dy\\
   =&\; c(n,s) \int_{\Rn \setminus B_r}  \frac{u(y)}{(|y|^2-1)^s} \lrq{ \frac{(-2sx_j) (1-|x|^2)^{s-1} }{|x-y|^n} + (-n) \frac{(1-|x|^2)^s (x_j-y_j)}{|x-y|^{n+2}} \, }dy.}
   Therefore renaming the constants (even from line to line),
   \eqlab{ \label{derivb}|Du(x)|\leq c_{n,s} \int_{\Rn\setminus B_1} \frac{|u(y)|}{(|y|^2-1)^s} \lrq{\frac{|x|(1-|x|^2)^{s-1} }{|x-y|^n} + \frac{(1-|x|^2)^s}{|x-y|^{n+1}} } \, dy.}
Given that $|x|\leq 1/2$ we have that $ 3 /4 \leq 1-|x|^2 \leq 1$ and $|x-y|\geq |y|/2$ and so 
   \bgs{|Du(x)|\leq c_{n,s}\|u\|_{L^\infty(\Rn\setminus B_1)}  \int_{\Rn\setminus B_1} \lrq{ \frac{1}{(|y|-1)^s |y|^n} +\frac{1}{(|y|-1)^s |y|^{n+1}} } dy.}
   Passing to polar coordinates and renaming the constants, we have that
   \bgs{|Du(x)|\leq c_{n,s}\|u\|_{L^\infty(\Rn\setminus B_1)} \lrq{ \int_1^\infty (\rho-1)^{-s}\rho^{-1}\, d\rho + \int_1^\infty (\rho-1)^{-s}\rho^{-2}\, d\rho}.}
 Now we compute
 \[\int_1^\infty (\rho-1)^{-s}\rho^{-1}\, d\rho =\int_1^{2} (\rho-1)^{-s}\rho^{-1}\, d\rho+\int_{2}^\infty (\rho-1)^{-s}\rho^{-1}\, d\rho \leq C \] and likewise,
  \[\int_1^\infty (\rho-1)^{-s}\rho^{-2}\, d\rho  \leq C .\]
  It follows that 
  \bgs{|Du(x)|\leq c_{n,s} \|u\|_{L^\infty(\Rn\setminus B_1)} \quad \mbox{ for any } x\in B_{1/2}.} By reiterating the computation, we obtain the conclusion for the $\alpha$ derivative. This proves the estimate \eqref{estimate16}, thus \eqref{estimate1} by rescaling.
 \end{proof}
 
 \begin{lemma}\label{estimate2} Let $f\in C^{0,\eps}( B_r)\cap C(\overline B_r)$ be a given function and $u\in C^1(B_r)\cap L^{\infty}(\Rn)$ be a pointwise solution of
\[\begin{cases}
		    \frlap u= f   \qquad &\mbox{ in } {B_r} , 
 		\\  u= 0   \qquad &\mbox{ in } {\Rn \setminus B_r}. 
	\end{cases}\] 
 					Then
 					\eqlab{\label{estimate2} \|u\|_{L^{\infty}(B_r)}\leq c r^{2s}\sup_{\overline B_r}|f|,}
 					where $c=c(n,s)$ is a positive constant. Furthermore, for $s>1/2$ 
 					\eqlab{\label{estimate3} \|Du\|_{L^{\infty}(B_{r/2})}\leq \overline c r^{2s-1} \sup_{\overline B_r}|f|,} where $\overline c=\overline c(n,s)$ is a positive constant.
  \end{lemma}
 \begin{proof}
  We notice that it is enough to prove \eqref{estimate2} and \eqref{estimate3} for $r=1$, i.e.
 \eqlab{\label{estimate26} \ \|u\|_{L^{\infty}(B_1)}\leq c \sup_{\overline B_1}|f|} and\eqlab{\label{estimate36} \|Du\|_{L^{\infty}(B_{1/2})}\leq \overline c \sup_{\overline B_1}|f|.}
    Indeed, by rescaling, we let $y=rx$ and  $v(y)=u(x)$ we have that  $\frlap v(y)=r^{-2s}\frlap u (x)$, while $r D v (y) =D u(x)$  and one gets the original estimates for any $r$.
    
 We take $\tilde f$ to be a continuous extension of $f$, namely let $\tilde f \in C^{0,\eps}_c(\Rn) $ be such that
 \sys[ \tilde f=]{ & f &\mbox{ in } &B_1 \\
 				&0 &\mbox{ in } &\Rn \setminus B_{3/2}} and $\sup_{\Rn}| \tilde f|\leq C \sup_{ \overline B_1} |f|$. 
 Let
 \eqlab{ \label{defu0} \tilde u(x):=\tilde f* \Phi(x) =a(n,s)\int_{\Rn} \frac{\tilde f(y)}{|x-y|^{n-2s}}\, dy.}
Then $\tilde u \in L_s^1(\Rn)$ (see Theorem 2.3 in \cite{bucur} for the proof) and  $\tilde u \in C^{2s+\eps}(\Rn)$, according to Lemma \ref{lem12}. Thanks to \eqref{fondsolcon}, we have that 
 $ \frlap \tilde u=\tilde f	$ pointiwse in $\Rn$. Hence, thanks to the definition of $\tilde f$,  $ \frlap (\tilde u-u)=0$ in $B_1$. Moreover, $\tilde u-u =\tilde u$ in $\Rn \setminus B_1$ and from \eqref{poissonrepr} we have that in $B_1$
\eqlab{\label{defu01}  (\tilde u-u)(x) =\int_{\Rn\setminus B_1} \tilde u(y)P_1(y,x)\, dy.}	We notice at first that  by definition \eqref{defu0} and passing to polar coordinates, we obtain for any positive constant $\tilde c$ that
 \eqlab{  \label{u01} \|\tilde u\|_{L^{\infty}(B_{\tilde c } )} \leq a_{n,s} \sup_{\Rn} |\tilde f| \int_0^{\tilde c +3/2}\rho^{2s-1}\, d\rho \leq c_{n,s}  \sup_{\overline B_1}|f| .}
By renaming constants, we also have that
 \eqlab{ \label{thing2} \|\tilde u-u\|_{L^{\infty}(B_1)} \leq &\;\int_{ B_{2}\setminus B_1} |\tilde u(y)| P_1(y,x)\, dy + \int_{\Rn\setminus   B_{2}} |\tilde u(y)| P_1(y,x)\, dy \\
 \leq &\; \|\tilde u\|_{L^{\infty}(B_{2})} + I\\
 \leq &\; c_{n,s}  \sup_{\overline B_1}|f|  + I.}
 Inserting the definition \eqref{poisson} and using for $|y|\geq 2$ the bounds $|y-x|\geq |y|/2$ and $|y|^2-1\geq |y|^2/2$ we have that
 \bgs{ I\leq &\;c(n,s)   \int_{\Rn\setminus   B_{2}}  \frac {|\tilde u(y)|}{(|y|^2-1)^s |x-y|^n}\, dy\\
 \leq &\; c_{n,s}    \int_{\Rn\setminus   B_{2}} \frac {|\tilde u(y)|}{|y|^{n+2s} }\, dy .} 
 We estimate the $L_s^1$ norm of $\tilde u$ as follows
 \eqlab{\label{thing6} \|\tilde u\|_{L_s^1(\Rn\setminus B_{2})} = &\; \int_{\Rn\setminus B_{2}} \frac{|\tilde u(y)|}{|y|^{n+2s}}\, dy\\
 \leq &\;a(n,s)\int_{\Rn\setminus B_{2}} |y|^{-n-2s} \lrq{ \int_{B_{3/2}} \frac{|\tilde f(t)|}{|y-t|^{n-2s}} \, dt} \,dy\\
 \leq &\; a(n,s)\sup_{\Rn}|\tilde f|  \int_{B_{3/2}}  \lrq{ \int_{\Rn\setminus B_{2}} |y|^{-n-2s}|y-t|^{2s-n}\, dy}\, dt.}
 We use that $|y-t|\geq |y|/4$ and passing to polar coordinates we get that
  \eqlab{\label{thing66} \|\tilde u\|_{L_s^1(\Rn\setminus B_{2})} \leq  a_{n,s}\sup_{\Rn}|\tilde f|    \int_{2}^\infty \rho^{-n-1} \, d\rho =a_{n,s} \sup_{\overline B_1}| f|  .}
Hence 
 \[ I \leq  c_{n,s}  \sup_{\overline B_1} |f|.\]
 It follows in \eqref{thing2} (eventually renaming the constants) that
\eqlab{ \|\tilde u-u\|_{L^{\infty}(B_1)} \leq  c_{n,s}\sup_{\overline B_1} |f|. \label{fb1} }

 By the triangle inequality, we have that 
 \bgs{ \|u\|_{L^{\infty}(B_1)} \leq \|\tilde u\|_{L^{\infty}(B_1)}  +\|\tilde u-u\|_{L^{\infty}(B_1)} .}
Hence by using \eqref{u01} and \eqref{fb1} we obtain that
\[ \|u\|_{L^{\infty}(B_1)} \leq c_{n,s} \sup_{\overline B_1} |f|,\]
that is the desired estimate \eqref{estimate26}, hence \eqref{estimate2} after rescaling.

In order to prove \eqref{estimate36}, we take $x\in B_{1/2}$ and obtain by the triangle inequality that
\eqlab{ \label{thing7} |Du(x)|\leq |D(\tilde u-u)(x)|+|D\tilde u(x)|.}
We notice that in the next computations the constants may change value from line to line.
 By using \eqref{defu01} and \eqref{derivb}, for $|x|\leq 1/2$ (hence $|y-x|\geq |y|/2$) we obtain that
 \eqlab{\label{thing4} |D(\tilde u-u)(x)|\leq &\; c_{n,s} \int_{\Rn\setminus B_1}\frac{|\tilde u(y)|}{(|y|^2-1)^s|y|^{n}}\, dy \\ &\; +c_{n,s} \int_{\Rn\setminus B_1}\frac{|\tilde u(y)|}{(|y|^2-1)^s|y|^{n+1}}\, dy \\
 =&\; c_{n,s} (I_1+I_2) .}
 We compute by passing to polar coordinates that
 \bgs{  \int_{B_{2}\setminus B_1}\frac{|\tilde u(y)|}{(|y|^2-1)^s|y|^{n}}\, dy  \leq c_{n,s} \|\tilde u\|_{L^{\infty}(B_{2 })} \leq c_{n,s} \sup_{\overline B_1} |f|, } according to \eqref{u01}
 Moreover, for $|y|\geq 2 $ we have that $|y|^2-1\geq |y|^2/2$ and so
 \bgs{ \int_{\Rn\setminus B_{2 }}\frac{|\tilde u(y)|}{(|y|^2-1)^s|y|^{n}}\, dy \leq   \int_{\Rn\setminus B_{2 }}\frac{|\tilde u(y)|}{|y|^{2s+n}}\, dy  \leq a_{n,s} \sup_{\overline B_1}|f|}
 thanks to \eqref{thing6} and \eqref{thing66}. Hence \[I_1\leq c_{n,s} \sup_{\overline B_1}|f|.\]
 We split also integral $I_2$ into two and by passing to polar coordinates, we get that
 \bgs{ \int_{B_{2 }\setminus B_1} \frac{|\tilde u(y)|}{(|y|^2-1)^s|y|^{n+1}}\, dy \leq &\; c_{n,s} \|\tilde u\|_{L^\infty(B_{2 })} \leq c_{n,s}   \sup_{\overline B_1}|f| } again by \eqref{u01}. 
 Also, using definition \eqref{defu0} of $\tilde u$ and for $|y|\geq2$ the fact that $|y|^2-1\geq |y|^2/2$, we get
 \[  \int_{\Rn \setminus B_{2 }} \frac{|\tilde u(y)|}{(|y|^2-1)^2|y|^{n+1}}\, dy  \leq a(n,s) \int_{\Rn \setminus B_{2 }}  |y|^{-n-2s-1} \lrq{\int_{B_{3/2}} \frac{|\tilde f(t)|}{|y-t|^{n-2s}} \, dt} \, dy.\] We have that $|y-t|\geq |y|/4$ and obtain that
 \[  \int_{\Rn \setminus B_{2 }} \frac{|\tilde u(y)|}{(|y|^2-1)^2|y|^{n+1}}\, dy  \leq a_{n,s}    \sup_{\overline B_1}|f|.\]
 It follows that 
 \bgs{ I_2\leq c_{n,s}   \sup_{\overline B_1}|f|.} Inserting the bounds on $I_1$ and $I_2$ into \eqref{thing4}, we finally obtain that
 \eqlab{\label{thing5} |D(\tilde u -u)(x)|\leq c_{n,s}  \sup_{\overline B_1}|f|.}
 On the other hand, for $s>1/2$, using \eqref{fundsolderiv} we get that
 \[ D\tilde u(x)=a(n,s) \int_{B_{3 /2}} \frac{\tilde f(y) }{|x-y|^{n-2s+1} }\, dy\] and therefore by passing to polar coordinates 
 \bgs{ |D\tilde u(x)|\leq &\;a_{n,s} \sup_{\overline B_1 }|f| \int_{B_{3 /2}} |x-y|^{2s-n-1}\, dy \leq a_{n,s} \sup_{\overline B_1 }|f|  \int_0^{2 } \rho^{2s-2}\, d\rho   \\&\;=a_{n,s}   \sup_{\overline B_1 } |f| .}
 This and \eqref{thing5} finally allow us to conclude from \eqref{thing7} that
 \[ |Du(x)|\leq \overline c   \sup_{\overline B_1} |f|\] for any $x\in B_{1/2}$, therefore the bound in \eqref{estimate36}. From this after rescaling, we obtain the  estimate in \eqref{estimate3}. 
   \end{proof}
% We use the representation formula given in \eqref{greenrepr} and obtain that
% \[ u(x)= \int_{B_r} f(y) G(x,y)\, dy,\]
% where $G$ is computed in \eqref{green}. 
% Then for any $x\in B_r$ we have that
% \bgs{ |u(x)|\leq&\;  \int_{B_r} |f(y)| G(x,y)\, dy \\
% \leq &\; \kappa(n,s)  \sup_{\overline B_r} |f| \int_{B^r}|x-y|^{2s-n} \int_0^{r_0(x,y)} \frac{t^{s-1}}{(t+1)^{n/2}} \, dt \, dy.}
% We use the inequality
% \[ \int_0^{r_0(x,y)} \frac{t^{s-1}}{(t+1)^{n/2}} \, dt\leq \int_0^{\infty} \frac{t^{s-1}}{(t+1)^{n/2}} \, dt =\kappa(n,s)^{-1},\] where we used the explicit computation given in formula of \cite{bucur}.
% It follows by passing to polar coordinates that
% \bgs{ |u(x)|\leq &\; \sup_{\overline B_r} |f|  \int_{B_r} |x-y|^{2s-n} \, dy\\
% \leq &\; c_n \sup_{\overline B_r} |f|  \int_0^{r+|x|} \rho^{2s-1} \, d\rho \\
% \leq &\; c_{n,s} \sup_{\overline B_r} |f|  r^{2s}.}
% This concludes the proof of the Lemma.

%--------------------------------------
% SECTION
%--------------------------------------
 \section{A proof of Schauder estimates} 
In this section we give a simple proof of some Schauder estimates related to the fractional Laplacian, as stated in Theorem \ref{schthm}.
%%We recall that a function is called Dini continuous in $B_1$ if  
%%\bgs{ \int_0^1 \frac{\omega(r)}{r} \, dr <\infty, \mbox{ where } \omega(r):=\sup_{|x-y|<r} |f(x)-f(y)|. } 
%\begin{theorem} \label{schthm}Let $s\in(0,1)$,  $\alpha <1$ and $f\in C^{0,\alpha}(B_1)\textcolor{red}{\cap C(\overline B_1)}$ be a given function with modulus of continuity $\omega(r):=\sup_{|x-y|<r} |f(x)-f(y)|$.
%% Dini continuous function in $B_1$ 
% Let  $u\in  L^{\infty}(\Rn)\cap C^1(B_1)$ be a pointwise solution of $\frlap u=f$ in $B_1$. Then for any $x,y \in B_{1/2}$ and denoting $\delta:=|x-y|$ we have that for $ s\leq 1/2$ 
%  \eqlab{\label{mm1}    |u(x)-u(y)|\leq &\; C_{n,s}  \bigg( \delta  \|u\|_{L^{\infty}(\Rn\setminus B_1)} +\delta \omega(1) + \int_0^{c \delta} \omega(t)t^{2s-1}\, dt + \delta \int_{\delta}^1  \omega(t)t^{2s-2}\, dt\bigg) }
%  while for $s>1/2$ 
%  \eqlab{\label{mm} 
% |Du(x)-Du(y)|\leq &\; C_{n,s}  \bigg(\delta  \|u\|_{L^{\infty}(\Rn\setminus B_1)} +\delta \omega(1) + \int_0^{c \delta} \omega(t)t^{2s-2}\, dt+ \delta \int_{\delta}^1  \omega(t)t^{2s-3}\, dt\bigg) ,}
% where $C_{n,s}$ and $ c$ are two positive constants.
%%  Moreover, if $f\in C^{0,\alpha}(B_1)$ for $\alpha <1$ we have that for any $x,y \in B_{1/2}$ and denoting $\delta:=|x-y|$ 
%% 	\[ |Du(x)-D u(y) |\leq  C_{n,s}\delta \lr{  \|u\|_{L^{\infty}(\Rn\setminus B_1)}   + \delta^{\alpha+2s-2}}.\]
%\end{theorem}
As we see by substituting in \eqref{mm1} and \eqref{mm} that $\omega(r)\leq C r^\alpha$, we obtain for $s\leq 1/2$ 
\[ |u(x)- u(y) |\leq  C_{n,s}\delta \lr{  \|u\|_{L^{\infty}(\Rn\setminus B_1)} + \sup_{\overline B_1}|f|   + \delta^{\alpha+2s-1}},\] 
hence $u\in C^{0,2s+\alpha}(B_{1/2})$ as long as $\alpha <1-2s$ and Lipschitz if $\alpha>1-2s$. For $s>1/2$ we have that
	\[ |Du(x)-D u(y) |\leq  C_{n,s}\delta \lr{  \|u\|_{L^{\infty}(\Rn\setminus B_1)} + \sup_{\overline B_1}|f|   + \delta^{\alpha+2s-2}}.\] Hence if $\alpha \leq 2-2s$ then $u\in C^{1,\alpha+2s-1}(B_{1/2})$ while for  $2-2s\leq \alpha <1$ the derivative $Du$ is Lipschitz in $B_{1/2}.$ 
The proof takes its inspiration from \cite{wang}, where a similar result is proved for the classical case of the Laplacian. 

We prove here the case $s>1/2$, noting that for $s\leq 1/2$ the proof follows in the same way, using the lower order estimates.
\begin{proof}[Proof of Theorem \ref{schthm}] For $k=1,2,\dots$, we denote by $B_k:=B_{\rho^{k}}(0)$, where $\rho =1/2$ and let $u_k$ be a solution of 
\sys{&\frlap u_k = f(0) &\mbox{ in } &B_k\\
	&u_k=u &\mbox{ in } &\Rn \setminus B_k.}
Then we have that
	\sys{&\frlap (u_k-u) = f(0)-f &\mbox{ in } &B_k\\
	&u_k-u=0 &\mbox{ in } &\Rn \setminus B_k.}
	We remark that in the next computations, the constants may change value from line to line. 
	
	Thanks to \eqref{estimate2}, we get that
	\eqlab{\label{my1} \|u_k-u\|_{L^{\infty}(B_k)} \leq &\;c_{n,s} \rho^{2ks}\sup_{B_k} |f(0)-f|\\
	\leq &\; c_{n,s} \rho^{2ks} \omega(\rho^k).}
		Using \eqref{estimate3}, we obtain that
	\eqlab{\label{my3} \|D (u_k-u)\|_{L^\infty(B_{k+1})} \leq c_{n,s} \rho^{(2s-1)k}\omega(\rho^k).} From here, sending $k$ to infinity, for $s>1/2$ it yields that
	\eqlab{ \label{lim1} \lim_{k \to \infty} D u_k(0) =D u(0).} 
Furthermore,
\sys{&\frlap (u_k-u_{k+1}) =0 &\mbox{ in } &B_{k+1}\\
	&u_k-u_{k+1}=u_k-u  &\mbox{ in } &B_k \setminus B_{k+1}\\
	&u_k-u_{k+1}=0 &\mbox{ in } &\Rn \setminus B_k,} hence from \eqref{estimate1} we have that
	\[ \|D(u_k-u_{k+1})\|_{L^{\infty}(B_{k+2})} \leq c_{n,s}  \rho^{-(k+1)} \sup_{B_k\setminus B_{k+1}} |u_k-u|\] and
	\bgs{ \|D^2(u_k-u_{k+1})\|_{L^{\infty}(B_{k+2})} \leq c_{n,s}  \rho^{-2(k+1)} \sup_{B_k\setminus B_{k+1}} |u_k-u|.}
	Using now \eqref{my1}, we get that
	\eqlab{\label{my2}\|D(u_k-u_{k+1})\|_{L^{\infty}(B_{k+2})}   \leq c_{n,s} \rho^{(2s-1)k}\omega(\rho^k)} and
\eqlab{\label{y1}\|D^2(u_k-u_{k+1})\|_{L^{\infty}(B_{k+2})} \leq c_{n,s}  \rho^{(2s-2)k} \omega(\rho^k).}
Let us fix $s>1/2$. Then for any given point $z$ near the origin we have that
\eqlab{\label{bla1} |Du(z)-Du(0)|\leq &\; |Du_k(z)-Du(z)|+ |Du_k(0)-Du(0)|+|Du_k(z)-Du_k(0)| \\ = &\;A_1+A_2+A_3.}
For $k \in \N^*$ fixed, we take $z$ such that $\rho^{k+2}\leq |z|\leq \rho^{k+1}$. Using \eqref{my3}  we get that
\[ A_1 \leq c_{n,s} \rho^{(2s-1)k}\omega(\rho^k).\] 
Taking into account \eqref{lim1} and using \eqref{my2}, we have that
\[ A_2 \leq \sum_{j=k}^{\infty} |Du_j(0)-D u_{j+1}(0)| \leq c_{n,s}\sum_{j=k}^{\infty} \rho^{(2s-1)j}\omega(\rho^j),\] therefore by renaming the constants
\bgs{  A_1+A_2 \leq &\;   c_{n,s} \rho^{(2s-1)k}\omega(\rho^k) + c_{n,s}\sum_{j=k}^{\infty} \rho^{(2s-1)j}\omega(\rho^j) \\ &\; \leq c_{n,s}\sum_{j=k}^{\infty} \rho^{(2s-1)j}\omega(\rho^j) .}  
 For the positive constant $c_s= (2s-1) / \lr{{\rho^{1-2s}-1}}$ and any $j =k, k+1, \dots$ we have that 
\[\rho^{(2s-1)j}=c_s \int_{\rho^{j}}^{\rho^{j-1}} t^{2s-2}\, dt.\]
Since $\omega$ is a increasing function, we obtain that
\bgs{\omega(\rho^j) \rho^{(2s-1)j}  = c_s\, \omega(\rho^j) \int_{\rho^{j}}^{\rho^{j-1}} t^{2s-2}\, dt \leq c_s \int_{\rho^{j}}^{\rho^{j-1}} \omega(t) t^{2s-2}\, dt,}
hence given that $8|z|\geq \rho^{k-1}$
\bgs{ \sum_{j=k}^{\infty} \rho^{(2s-1)j}\omega(\rho^j) \leq &\;  c_s \sum_{j=k}^{\infty} \int_{\rho^{j}}^{\rho^{j-1}} \omega(t) t^{2s-2}\, dt
\leq c_s \int_0^{\rho^{k-1}} \omega(t) t^{2s-2}\, dt
% \leq   &\;c_s \lr{ \int_0^{|z|} \omega(t)t^{2s-2}\, dt + \int_{\rho^{k+2}}^{\rho^{k-1}}\omega(t)t^{2s-2}\, dt  }
 \\
 \leq  &\; {c_s  \int_0^{8|z|} \omega(t)t^{2s-2}\, dt } .}
  Therefore, \eqlab{\label{a12} A_1 + A_2 \leq c_{n,s} \int_0^{8|z|} \omega(t)t^{2s-2}\, dt.} 
%
%By a similar argument we can estimate $A_2$, through the solutions of 
%$(-\Delta)^s v_j = f(z)$ in $B_j(z)$ and $v_j = u$ on $\R^n\setminus B_j(z)$ for $j=k, k+1,\dots.$ \textcolor{red}{Then,  
%\[ |D v_k(z) - D u(z) |\leq c_{n,s}\sum_{j=k}^{\infty} \rho^{(2s-1)j}\omega(\rho^j) .\] Moreover, 
%\[ |Dv_k(z) -Du_k(z) | \leq c_{n,s} \rho^{(k-1)(2s-1)} \omega(\rho^{k-1}),\]
%as a consequence of estimate  in \eqref{estimate3} and the fact that $v_k-u_k$ vanishes on $\Rn \setminus B_{k-1}$.
%Hence, 
%\bgs{ A_2 \leq &\; \sum_{j=k-1}^{\infty} \rho^{(2s-1)j}\omega(\rho^j) \leq c_s \int_0^{\rho^{k-2}} \omega(t) t^{2s-2}\, dt \\ 
%\leq  &\; {c_s  \int_0^{16|z|} \omega(t)t^{2s-2}\, dt }, }
% given that $16|z|\geq \rho^{k-2}.$
%
%Consequently, we have that \[A_2\leq c_{n,s} \int_0^{16|z|} \omega(t)t^{2s-2}\, dt.\]}
Moreover,  for $j=0, 1, \dots, k-1$ we consider $h_j:= u_{j+1}-u_{j}$ and have that
	\bgs{ A_3 \leq \sum_{j=0}^{k-1} |Dh_j(z)-Dh_j(0)| + |Du_0(z)-Du_0(0)|.}
	By the mean value theorem, there exists $\theta \in(0,|z|)$ such that
	\[ |Dh_j(z)-Dh_j(0)|\leq |z||D^2h_j(\theta)|\] and since $|z|\leq \rho^{k+1}$, thanks to \eqref{y1} we obtain that 
	\bgs{ |D^2h_j(\theta)| \leq  c_{n,s} \rho^{(2s-2)j} \omega(\rho^{j}).}
	Hence
	\[\sum_{j=0}^{k-1}|Dh_j(z)-Dh_j(0)| \leq c_{n,s} |z|  \sum_{j=0}^{k-1} \rho^{(2s-2)j}\omega(\rho^{j})
	=c_{n,s} |z|  \Big( \sup_{\overline B_1}|f| +  \sum_{j=1}^{k-1} \rho^{(2s-2)j}\omega(\rho^{j})\Big).\]
As previously done, we have that for the positive constant $c_s=(2-2s)/\lr{{1-\rho^{2-2s}}}$  and $j=1,\dots, k-1$
\[\rho^{(2s-2)j}= c_s \int_{\rho^{j}}^{\rho^{j-1}} t^{2s-3}\, dt\]and since $\omega$ is increasing
\bgs{\omega(\rho^{j}) \rho^{(2s-2)j} \leq c_s \int_{\rho^{j}}^{\rho^{j-1}} \omega(t) t^{2s-3}\, dt.} It follows that
\bgs{ \sum_{j=1}^{k-1} \rho^{(2s-2)j}\omega(\rho^{j}) \leq &\;  c_s \sum_{j=1}^{k-1} \int_{\rho^{j}}^{\rho^{j-1}} \omega(t) t^{2s-3}\, dt
\leq c_s  \int_{\rho^{k-1}}^1 \omega(t)t^{2s-3}\, dt\\
 \leq   &\;c_s \int_{|z|}^1 \omega(t)t^{2s-3}\, dt, } since $|z|\leq \rho^{k-1}$.  Therefore, \[\sum_{j=1}^{k-1}|Dh_j(z)-Dh_j(0)| \leq c_{n,s} |z| \int_{|z|}^1  \omega(t)t^{2s-3}\, dt.\]
 Moreover, let 
 \[ v_0(x):= k_{n,s} f(0) (1-|x|^2)^s_+ \quad \mbox{ for } x \in \Rn . \]  Then (see Section 2.6 in \cite{nonlocal}, or the general result in \cite{dyda}), for the appropriate value of $ k_{n,s}$, we have in $B_1$ that  $\frlap v_0(x) =f(0)$. Then the function $u_0-v_0$ is $s$-harmonic in $B_1$, with boundary data $u$. We have that
 \[ |Du_0(z)-Du_0(0)|\leq |z| |D^2 u_0(\theta)|\leq |z|\lr{ |D^2(u_0-v_0) (\theta)| + |D^2 v_0(\theta)|}.\]
  Using the estimate in \eqref{estimate1} we have  for $\theta \in (0,|z|)$
    \[ |D^2(u_0-v_0) (\theta)|\leq  c_{n,s}\,  \, \|u\|_{L^{\infty}(\Rn\setminus B_1)}.\] Moreover, $|D^2 v_0(\theta)|$ is bounded. It follows that  
    \[ |Du_0(z)-Du_0(0)|\leq c_{n,s} |z| \|u\|_{L^{\infty}(\Rn\setminus B_1) } ,\] hence
 \[ A_3 \leq c_{n,s}|z| \lr{ \sup_{\overline B_1}|f| + \|u\|_{L^{\infty}(\Rn\setminus B_1) } + \int_{|z|}^1  \omega(t)t^{2s-3}\, dt}.\]
 
 Inserting this and \eqref{a12} into \eqref{bla1} we finally obtain that 
 \bgs{ |Du(z)-Du(0)| \leq &\; C_{n,s} \bigg[ |z|\lr{  \|u\|_{L^{\infty}(\Rn\setminus B_1)} +\sup_{\overline B_1}|f|}  + \int_0^{c|z|} \omega(t)t^{2s-2}\, dt \\  &\;+ |z| \int_{|z|}^1  \omega(t)t^{2s-3}\, dt \bigg].} From this the conclusion plainly follows.
%  If furthermore $f \in C^{0,\alpha}(B_1)$  by integrating we obtain that
%  \[ |Du(z)-D u(0)|\leq C_{n,s} \lrq{ |z|\|u\|_{L^{\infty}(\Rn\setminus B_1)} + |z|^{\alpha+2s-1} }. \]
  This concludes the proof of the Theorem.
\end{proof}
%
%\bibliography{biblio}

\begin{thebibliography}{10}

\bibitem{bucur}
Claudia Bucur.
\newblock Some observations on the {G}reen function for the ball in the
  fractional {L}aplace framework.
\newblock {\em Communications on Pure and Applied Analysis}, 15(2):657--699,
  2016.

\bibitem{nonlocal}
Claudia Bucur and Enrico Valdinoci.
\newblock Nonlocal diffusion and applications.
\newblock {\em arXiv preprint arXiv:1504.08292}, 2015.
\newblock To appear in the Springer Series ``Unione Matematica Italiana''.

\bibitem{Constantin}
Peter Constantin, Andrew~J. Majda, and Esteban Tabak.
\newblock Formation of strong fronts in the 2-d quasigeostrophic thermal active
  scalar.
\newblock {\em Nonlinearity}, 7(6):1495--1533, 1994.

\bibitem{galattica}
Eleonora Di~Nezza, Giampiero Palatucci, and Enrico Valdinoci.
\newblock Hitchhiker's guide to the fractional {S}obolev spaces.
\newblock {\em Bull. Sci. Math.}, 136(5):521--573, 2012.

\bibitem{samko}
Lars Diening and Stefan~G Samko.
\newblock On potentials in generalized {H}{\"o}lder spaces over uniform domains
  in $\mathbb{R}^n$.
\newblock {\em Revista matem{\'a}tica complutense}, 24(2):357--373, 2011.

\bibitem{Dong}
Hongjie Dong and Doyoon Kim.
\newblock Schauder estimates for a class of non-local elliptic equations.
\newblock {\em Discrete Contin. Dyn. Syst.}, 33, 2013.

\bibitem{dyda}
Bart{\l}omiej Dyda.
\newblock Fractional calculus for power functions and eigenvalues of the
  fractional {L}aplacian.
\newblock {\em Fract. Calc. Appl. Anal.}, 15(4):536--555, 2012.

\bibitem{fall}
Mouhamed~Moustapha Fall.
\newblock Entire $ s $-harmonic functions are affine.
\newblock {\em arXiv preprint arXiv:1407.5934}, 2014.

\bibitem{trudy}
David Gilbarg and Neil~S. Trudinger.
\newblock {\em Elliptic partial differential equations of second order}.
\newblock Classics in Mathematics. Springer-Verlag, Berlin, 2001.
\newblock Reprint of the 1998 edition.

\bibitem{Landkof}
N.~S. Landkof.
\newblock {\em Foundations of modern potential theory}.
\newblock Springer-Verlag, New York-Heidelberg, 1972.
\newblock Translated from the Russian by A. P. Doohovskoy, Die Grundlehren der
  mathematischen Wissenschaften, Band 180.

\bibitem{oton}
Xavier Ros-Oton and Joaquim Serra.
\newblock The {D}irichlet problem for the fractional {L}aplacian: regularity up
  to the boundary.
\newblock {\em Journal de Math{\'e}matiques Pures et Appliqu{\'e}es},
  101(3):275--302, 2014.

\bibitem{Valdy}
Raffaella Servadei and Enrico Valdinoci.
\newblock Weak and viscosity solutions of the fractional {L}aplace equation.
\newblock {\em Publ. Mat.}, 58(1):133--154, 2014.

\bibitem{Silvestre}
Luis Silvestre.
\newblock Regularity of the obstacle problem for a fractional power of the
  {L}aplace operator.
\newblock {\em Comm. Pure Appl. Math.}, 60(1):67--112, 2007.

\bibitem{wang}
Xu-Jia Wang.
\newblock Schauder estimates for elliptic and parabolic equations.
\newblock {\em Chin. Ann. Math.}, 27B:637--642, 2006.

\bibitem{russo}
Victor~Iosifovich Yudovich.
\newblock Non-stationary flows of an ideal incompressible fluid.
\newblock {\em Zhurnal Vychislitel'noi Matematiki i Matematicheskoi Fiziki},
  3(6):1032--1066, 1963.

\end{thebibliography}
%\bibliographystyle{plain}

\end{document}